\documentclass[11pt]{article}
\usepackage[english]{babel}
\usepackage[latin1]{inputenc}
\usepackage{amsmath}
\usepackage{amssymb}
\usepackage{amsfonts}
\usepackage{amsthm}
\usepackage{color}

\newcommand{\p}{\mathfrak{p}}
\newcommand{\q}{\mathfrak{q}}

\renewcommand{\v}{\mathfrak{v}}
\newcommand{\C}{\mathbb{C}}
\newcommand{\E}{\mathfrak{E}}
\newcommand{\N}{\mathbb{N}}

\newcommand{\R}{\mathbb{R}}

\newcommand{\Z}{\mathbb{Z}}
\newcommand{\boA}{\mathcal{A}}
\newcommand{\boC}{\mathcal{C}}

\newcommand{\boE}{\mathcal{E}}

\newcommand{\boP}{\mathcal{P}}

\newcommand{\boX}{\mathcal{X}}
\newcommand{\boZ}{\mathcal{Z}}
\newcommand{\Card}{{\rm Card}}

\renewcommand{\th}{{\rm th}}

\newtheorem{claim}{Claim}
\newtheorem{cor}{Corollary}
\newtheorem{lemma}{Lemma}
\newtheorem{prop}{Proposition}
\newtheorem{step}{Step}
\newtheorem{theorem}{Theorem}
\theoremstyle{definition}
\newtheorem*{merci}{Acknowledgements}
\newtheorem{remark}{Remark}

\begin{document}

\title{Orbital stability of the black soliton to the Gross-Pitaevskii equation}
\author{\renewcommand{\thefootnote}{\arabic{footnote}}
Fabrice B\'ethuel \footnotemark[1], Philippe Gravejat \footnotemark[2], Jean-Claude Saut \footnotemark[3], Didier Smets \footnotemark[4]}
\footnotetext[1]{Laboratoire Jacques-Louis Lions, Universit\'e Pierre et Marie Curie, Bo\^ite Courrier 187, 75252 Paris Cedex 05, France. E-mail: bethuel@ann.jussieu.fr}
\footnotetext[2]{Centre de Recherche en Math\'ematiques de la D\'ecision, Universit\'e Paris Dauphine, Place du Mar\'echal De Lattre De Tassigny, 75775 Paris Cedex 16, France. E-mail: gravejat@ceremade.dauphine.fr}
\footnotetext[3]{Laboratoire de Math\'ematiques, Universit\'e Paris Sud, B\^atiment 425, 91405 Orsay Cedex, France. E-mail: Jean-Claude.Saut@math.u-psud.fr}
\footnotetext[4]{Laboratoire Jacques-Louis Lions, Universit\'e Pierre et Marie Curie, Bo\^ite Courrier 187, 75252 Paris Cedex 05, France. E-mail: smets@ann.jussieu.fr}
\date{April 1, 2008}
\maketitle

\begin{abstract}
We establish the orbital stability of the black soliton, or kink solution, $\v_0(x) = \th(\frac{x}{\sqrt{2}})$, to the one-dimensional Gross-Pitaevskii equation, with respect to perturbations in the energy space.
\end{abstract}

\section{Introduction}

In this paper, we consider the one-dimensional Gross-Pitaevskii equation
\renewcommand{\theequation}{GP}
\begin{equation}
\label{GP}
i \Psi_t + \Psi_{xx} = \Psi (|\Psi|^2-1) \ {\rm on} \ \R \times \R,
\end{equation}
which is a version of the defocusing cubic nonlinear Schr\"odinger equations. We supplement this equation with the boundary condition at infinity
\renewcommand{\theequation}{\arabic{equation}}
\numberwithin{equation}{section}
\setcounter{equation}{0}
\begin{equation}
\label{bdinfini}
|\Psi(x, t)| \to 1, \ {\rm as} \ |x| \to + \infty.
\end{equation}
This boundary condition is suggested by the formal conservation of the energy (see \eqref{GLE} below), and by the use of the Gross-Pitaevskii equation as a physical model, e.g. for the modelling of ``dark solitons'' in nonlinear optics (see \cite{KivsLut1}). Moreover, the boundary condition \eqref{bdinfini} provides \eqref{GP} with a richer dynamics than in the case of null condition at infinity which is essentially governed by dispersion and scattering. In particular, equation \eqref{GP} with condition \eqref{bdinfini} has nontrivial localized coherent structures called ``solitons''.

At least on a formal level, the Gross-Pitaevskii equation is hamiltonian. The conserved Hamiltonian is a Ginzburg-Landau energy, namely
\begin{equation}
\label{GLE}
E(\Psi) = \frac{1}{2} \int_{\R} |\Psi'|^2 + \frac{1}{4} \int_{\R} (1 - |\Psi|^2)^2 \equiv \int_{\R} e(\Psi).
\end{equation}
Similarly, as far as it might be defined, the momentum
$$P(\Psi) = \frac{1}{2} \int_{\R} \langle i \Psi, \Psi' \rangle$$
is formally conserved. We will see though that the definition of this quantity raises a number of difficulties. Another quantity which is formally conserved by the flow is the mass
$$m(\Psi) = \frac{1}{2} \int_{\R} \Big( |\Psi|^2 - 1 \Big).$$

In this paper, we will only consider {\bf finite energy} solutions to \eqref{GP}.
Equation \eqref{GP} is then integrable in dimension one by means of the inverse scattering method, and it has been formally analyzed within this framework in \cite{ShabZak2}. Recently, P. G\'erard and Z. Zhang \cite{GeraZha1} gave a complete justification of the method, obtaining in particular rigorous results on the Cauchy problem.

Stationary solutions to \eqref{GP}, that is time independent solutions, are of the form
$$\Psi(x,t) = u(x), \ \forall t \in \R,$$
where the profile $u$ solves the ordinary differential equation
\begin{equation}
\label{ST}
u_{xx} + u (1 - |u|^2) = 0.
\end{equation}
Equation \eqref{ST} may be integrated using standard arguments from ordinary differential equation theory. The non-constant solution of finite energy to \eqref{ST} is given, up to the invariances, by
$$\v_0(x) = \th \Big( \frac{x}{\sqrt{2}} \Big),$$
that is any non-constant solution to \eqref{ST} is of the form
$$u(x) = \exp i \theta \ \v_0(x - a),$$
where $a$ and $\theta$ are arbitrary real numbers. Notice that $\v_0$ is real-valued and vanishes at the origin. Moreover, it converges exponentially fast to $\pm 1$, as $x \to \pm \infty$. This stationary solution is known as a "black soliton" in nonlinear optics (see \cite{KivsLut1}), and is often termed a kink solution. It plays an important role in the theory of phase transitions.

The purpose of this paper is to establish the {\bf orbital stability} of the kink solution. Notice that Di Menza and Gallo \cite{diMeGal1} proved the linear stability and performed several numerics which suggest that a stronger notion of stability does hold.

To state our result, we first recall the classical notion of orbital stability (see, for instance, \cite{Benjami1}). The solution $\v_0$ is said to be orbitally stable in the metric space $X$, if and only if given any $\varepsilon > 0$, there exists some $\delta > 0$ such that for any solution $\Psi$ to \eqref{GP} in $X$, if
$$d_X \big( \Psi(\cdot , 0), \v_0 \big) \leq \delta,$$
then
$$\sup_{t \in \R} \bigg( \inf_{(a, \theta) \in \R^2} d_X \big( \Psi(\cdot, t), \exp i \theta \v_0(\cdot - a ) \big) \bigg) \leq \varepsilon.$$
As a preliminary step, this definition requires to prove that the Cauchy problem for \eqref{GP} is globally well-posed in $X$. A natural choice for $X$ is the energy space
$$\boX^1 \equiv \{ w \in L^\infty(\R), \ {\rm s.t.} \ w' \in L^2(\R), 1 - |w|^2 \in L^2(\R) \}.$$
Given any $v_0 \in \boX^1$, Zhidkov \cite{Zhidkov1} (see also \cite{Gerard2}) established that \eqref{GP} has a global solution with initial data $v_0$. More precisely, we have

\begin{theorem}[\cite{Zhidkov1,Gerard2}]
\label{cauchy}
Let $v_0 \in \boX^1$. There exists a unique solution $v$ of $(GP)$ such
that $v(0) = v_0$, and $t \mapsto v(t) - v_0 \in \boC^0(\R, H^1(\R))$. Moreover, the Ginzburg-Landau energy is conserved,
$$E(v(t)) = E(v_0), \ \forall t \in \R.$$
\end{theorem}

Given any $A > 0$, we consider on $\boX^1$ the distance $ d_{A, \boX^1}$ defined by
$$d_{A, \boX^1}(v_1, v_2) \equiv \| v_1 - v_2 \|_{L^\infty([- A, A])} + \| v_1' - v_2' \|_{L^2(\R)} + \| |v_1| - |v_2| \|_{L^2(\R)}.$$
Our main result is

\begin{theorem}
\label{stationnaire}
Assume that $v_0 \in \boX^1$ and consider the global in time solution
$v$ to \eqref{GP} with initial datum $v_0$. Given any numbers $\varepsilon > 0$ and
$A > 0$, there exists some positive number $\delta$, such that if
\begin{equation}
\label{pioneer}
d_{A, \boX^1}(v_0, \v_0) \leq \delta,
\end{equation}
then, for any $t \in \R$, there exist numbers $a(t)$ and $\theta(t)$ such that
\begin{equation}
\label{solaris}
d_{A, \boX^1} \big( v(\cdot + a(t), t), \exp i \theta (t) \ \v_0(\cdot) \big) < \varepsilon.
\end{equation}
\end{theorem}

The number $a(t)$ provided by $\eqref{solaris}$ is not unique. It describes the shift in space of the solution. We will show that $a(t)$ moves slowly. More precisely, we have

\begin{theorem} 
\label{vitesse}
Given any numbers $\varepsilon > 0$, sufficiently small, and $A > 0$, there exists some constant $K$, only depending on $A$, and some positive number $\delta > 0$ such that, if $v_0$ and $v$ are as in Theorem \ref{stationnaire} and if \eqref{pioneer} holds, then
\begin{equation}
\label{shift}
|a(t)| \leq K \varepsilon (1 + |t|),
\end{equation}
for any $t \in \R$, and for any of the points $a(t)$ satisfying inequality \eqref{solaris} for some $\theta(t) \in \R$.
\end{theorem}

In other words, Theorem \ref{vitesse} shows that the speed of the shift is smaller than any arbitrary positive number, provided we are sufficiently close to the kink solution.

\begin{remark}
It would be of interest to obtain a similar piece of information for the quantity $\theta(t)$.
\end{remark}

\begin{remark}
P. G\'erard and Z. Zhang \cite{GeraZha1} have used the inverse scattering method to prove a more precise version of Theorem \ref{stationnaire} in the $L^\infty$-norm, but for a more restricted class of perturbations.
\end{remark}

It is worthwhile mentioning that the stationary solution $\v_0$ belongs to a branch of more general special solutions, namely the branch of travelling wave solutions. Travelling waves are solutions to \eqref{GP} of the form
$$\Psi(x,t) = u(x - c t).$$
Here, the parameter $c \in \R$ corresponds to the speed of the travelling waves (we may restrict to the case $c \geq 0$ using complex conjugation). The case $c > 0$ corresponds to the "gray solitons" in nonlinear optics (see \cite{KivsLut1}). The equation for the profile $u$ is given by
\renewcommand{\theequation}{TWc}
\begin{equation}
\label{TWc}
-i c u_x + u_{xx} + u (1 - |u|^2) = 0.
\end{equation}
Equation \eqref{TWc} is entirely integrable using standard arguments from ordinary differential equation theory. The sound velocity $c_s = \sqrt{2}$ appears naturally in the hydrodynamics formulation of \eqref{GP} (see \cite{BetGrSa2}). If $|c| \geq \sqrt{2}$, $u$ is a constant of modulus one, whereas if $- \sqrt{2} < c < \sqrt{2}$, then, up to a multiplication by a constant of modulus one and a translation, $u$ is either identically equal to $1$, or to
\renewcommand{\theequation}{\arabic{equation}}
\numberwithin{equation}{section}
\setcounter{equation}{6}
\begin{equation}
\label{schumi}
u(x) = \v_c(x) \equiv \sqrt{\frac{2 - c^2}{2}} \th \Big( \frac{\sqrt{2 - c^2}}{2} x \Big)+ i \frac{c}{\sqrt{2}}.
\end{equation}
The non-constant travelling waves form a smooth branch of subsonic solutions to \eqref{TWc}. As seen before, there exist neither sonic, nor supersonic non-constant travelling waves.

In view of identity \eqref{schumi}, we also observe that $\v_c(x)$ does not vanish unless $c = 0$. Moreover, formula \eqref{schumi} yields the spatial asymptotics of the non-constant solutions to \eqref{TWc}. Notice in particular that
$$\v_c(x) \to \v_c^{\pm \infty} \equiv \pm \sqrt{1 - \frac{c^2}{2}} + i \frac{c}{\sqrt{2}}, \ {\rm as} \ x \to \pm \infty.$$
Hence, $\v_c(x)$ converges to a constant $\v_c^{\pm \infty}$ of modulus one, as $x \to \pm \infty$, the limits in $- \infty$ and $+ \infty$ being distinct. Notice also that the function $\v_c - \v_c^{\pm \infty}$ has exponential decay at infinity.

Orbital stability of travelling waves for any $- \sqrt{2} < c < \sqrt{2}$, was established in \cite{LinZhiw1} using a method of \cite{GriShSt1}, and later in \cite{BetGrSa2} relying on a variational principle, combined with several conservation laws. As a matter of fact, travelling wave solutions can be identified with critical points of the energy, keeping the momentum fixed. In this variational interpretation of the equation, the speed $c$ appears as the Lagrange multiplier related to the constraint, i.e. keeping the momentum fixed. The solution $\v_c$ corresponds to minima of the energy for fixed momentum. Our aim here is to extend this variational argument to the case $c = 0$.

The precise mathematical definition of the momentum raises however a serious difficulty, in particular because $\v_0$ vanishes. Recall that in the context of nonlinear Schr\"odinger equations, the momentum of maps $v$ from $\R$ to $\C$ should be defined as
$$P(v) = \frac{1}{2} \int_\R \langle i v, v' \rangle.$$
This quantity is not well-defined for arbitrary maps in the energy space.

To get convinced of this fact, assume that the map $v$ has the form $v = \exp i \varphi$, where $\varphi$ is real-valued, so that $|v| = 1$. For such a map,
$$\langle i v, v' \rangle = \varphi',$$
and the fact that $v$ belongs to the energy space $\boX^1$ is equivalent to the fact that $\varphi'$ belongs to $L^2(\R)$. We then have
$$P(v) = \frac{1}{2} \int_{\R} \varphi' = \frac{1}{2} \big[ \varphi \big]_{- \infty}^{+ \infty} \equiv \frac{1}{2} \big(\varphi(+ \infty) - \varphi(- \infty) \big),$$
which has a meaning if the map $v$ belongs to
$$\boZ^1 = \Big\{ v \in \boX^1, \ {\rm s.t.} \ v_{\pm \infty} = \lim_{x \to \pm \infty} v(x) \ {\rm exist} \Big\},$$ 
but not for any arbitrary phase $\varphi$ whose gradient is in $L^2$.

More generally, given any map $v$ in $\tilde{\boX^1}$, where $\tilde{\boX^1}$ is defined by
$$\tilde{\boX^1} = \{ v \in \boX^1, \ {\rm s.t.} \ |v(x)| >0, \ \forall x \in \R \},$$
we may write $v = \varrho \exp i \varphi$, so that $\langle i v, v' \rangle=\varrho^2\varphi'$. If $v \in \tilde{\boZ^1} \equiv \tilde{\boX^1} \cap \boZ^1$, we are led to
\begin{equation}
\label{goret}
P(v) = \frac{1}{2} \int_\R \varrho^2 \varphi' = \frac{1}{2} \int_\R (\varrho^2 - 1) \varphi' + \frac{1}{2} \int_\R \varphi' = \frac{1}{2} \int_\R (\varrho^2 - 1) \varphi'+ \frac{1}{2} \big[ \varphi \big]_{- \infty}^{+ \infty}.
 \end{equation}
Since $e(v) = \frac{1}{2} (\varrho'^2 + \varrho^2 \varphi'^2) + \frac{1}{4} (1 - \varrho^2)^2$, the first integral may be bounded by Cauchy-Schwarz inequality,
$$\bigg| \int_\R (\varrho^2 - 1) \varphi' \bigg| \leq \frac{1}{2} \int_\R (1 - \varrho^2)^2 + \frac{1}{2 \delta^2} \int_\R \varrho^2 (\varphi')^2 \leq \frac{2}{\delta^2} E(v),$$
where $\delta = \inf\{ |v(x)|, \ x \in \R\} > 0$, so that $P(v)$ is well-defined by formula \eqref{goret}.

It remains to give a meaning to the momentum $P(v)$ for maps having possibly zeroes. We define the momentum $P(v)$ of maps $v$ in $\boZ^1$ as follows.

\begin{lemma}
\label{limite}
Let $v \in \boZ^1$. Then, the limit 
$$\boP(v) = \underset{R \to + \infty}{\lim} P_R(v) \equiv \lim_{R\to +\infty} \int_{- R}^{R} \langle i v, v' \rangle$$
exists. Moreover, if $v$ belongs to $\tilde{\boZ^1}$, then
$$\boP(v) = \frac{1}{2} \int_\R (\varrho^2 - 1) \varphi'+ \frac{1}{2} \big[ \varphi \big]_{- \infty}^{+ \infty}.$$
\end{lemma}

In other words, using the definition provided by Lemma \ref{limite}, we have defined the momentum as an {\bf improper} integral. It offers however a sound mathematical formulation of the momentum $P$, at least if one restricts oneself to the space $\boZ^1$. We illustrate our previous constructions with the kink solution $\v_0$. Since $\v_0$ is real-valued, we have $\langle i \v_0, \v_0' \rangle = 0$, so that
$$\boP(\v_0) = 0.$$
In several computations, we use the following elementary observation.

\begin{lemma}
\label{andouille}
Let $V_0 \in \boZ^1$ and $w \in H^1(\R)$. Then, $V_0 + w \in \boZ^1$ and
\begin{equation}
\label{andouille1}
\boP(V_0 + w) = \boP(V_0) + \frac{1}{2} \int_\R\langle i w, w' \rangle + \int_\R \langle i w, V_0' \rangle.
\end{equation}
\end{lemma}

Notice that the right-hand side of identity \eqref{andouille1} involves, besides $\boP(V_0)$, only {\bf definite} integrals: this is an important advantage for establishing the corresponding conservation laws.

Another quantity which plays an important role in the variational formulation of \eqref{TWc} for $c \neq 0$ is the renormalized momentum $p$, which is defined for $v \in \tilde{\boX_1}$ by
\begin{equation}
\label{defnormmo}
p(v) = \frac{1}{2} \int_\R (\varrho^2 - 1) \varphi',
\end{equation}
so that, as seen before, if $v$ belongs to $\tilde{\boZ^1}$, then,
$$p(v) = \boP(v) - \frac{1}{2} \big[ \varphi \big]_{- \infty}^{+ \infty}.$$
If $v \in \boZ^1 \setminus \tilde{\boZ^1}$, the right-hand side of \eqref{defnormmo} is a priori not well-defined since the phase $\varphi$ is not globally defined. Nevertheless, the argument $\arg{v}$ of $v$ is well-defined at infinity as an element of the quotient space $\R/ 2\pi\Z$. Given $v \in \boZ^1$, we are therefore led to introduce the untwisted momentum
$$[p](v) = \Big( \boP(v) - \frac{1}{2} \big( \arg{v(+ \infty)} - \arg{v(- \infty)} \big) \Big) \ {\rm mod} \ \pi,$$
which is hence an element of $\R/ \pi\Z$. A remarkable fact concerning $[p]$ is that its definition extends to the whole space $\boX^1$, although for arbitrary maps in $\boX^1$, the quantity $\arg{v(+\infty)}-\arg{v(-\infty)}$ may not exist. Indeed, we have

\begin{lemma}
\label{limite2}
Assume that $v$ belongs to $\boX^1$. Then the limit
$$[p](v) = \lim_{R \to +\infty} \bigg( \int_{- R}^R \langle i v, v' \rangle - \frac{1}{2} \big( \arg{v(R)} - \arg{v(-R)} \big) \bigg) \ {\rm mod} \ \pi$$
exists. Moreover, if $v$ belongs to $\tilde{\boX^1}$, then
\begin{equation}
\label{bacon2}
[p](v) = p(v) \ {\rm mod} \ \pi.
\end{equation}
\end{lemma}

Similarly to Lemma \ref{andouille}, we have

\begin{lemma}
\label{lem:andouille2}
Let $V_0 \in \boX^1$ and $w \in H^1(\R)$. Then, $V_0 + w \in \boX^1$ and
\begin{equation}
\label{eq:andouille2}
[p](V_0 + w) = [p](V_0) + \frac{1}{2} \int_\R \langle i w, w' \rangle + \int_\R \langle i w, V_0' \rangle \ {\rm mod} \ \pi.
\end{equation}
\end{lemma}

In some places, when this does not lead to a confusion, we will identify elements of $\R / \pi \Z$ with their unique representative in the interval $]-\frac{\pi}{2}, \frac{\pi}{2}]$. 

In \cite{BetGrSa2}, the following minimization problem
\begin{equation}
\label{cochette1}
E_{\min}(\p) = \inf\{ E(v), v \in \tilde{\boX^1} \ {\rm s.t.} \ p(v) = \p \},
\end{equation}
was considered and solved for any $\p \in [0, \frac{\pi}{2})$.

\begin{lemma}
\label{tildeemin}
The function $\p \mapsto E_{\min}(\p)$ is non-decreasing and concave on $[0, \frac{\pi}{2})$. Moreover, given any $0 \leq \p < \frac{\pi}{2}$, problem \eqref{cochette1} is achieved by a unique minimizer, up to the invariances of the problem, the map $\v_{c(\p)}$. Here, $c(\p)$ denotes the unique speed $c$ such that $p(\v_c) = \p$.
\end{lemma}

\begin{remark}
Since the map $\v_c$ has no zero for $0 < c < \sqrt{2}$, and hence belongs to the space $\tilde{\boX^1}$, we may compute its momentum $p(\v_c)$. A short computation yields
$$p(\v_c) = \frac{\pi}{2} - \arctan \Big( \frac{c}{\sqrt{2 - c^2}} \Big) - \frac{c}{2} \sqrt{2 - c^2},$$
whereas the energy of $\v_c$ is equal to
\begin{equation}
\label{piquet}
E(\v_c) = \frac{(2 - c^2)^\frac{3}{2}}{3}.
\end{equation}
We notice that the function $c \mapsto p(\v_c)$ is smooth, decreasing, and satisfies 
\begin{equation}
\label{monaco}
\frac{\rm d}{\rm dc} p(\v_c) = - \sqrt{2 - c^2}.
\end{equation}
Hence, it performs a diffeomorphism from $(0, \sqrt{2})$ on $(0, \frac{\pi}{2})$, so that there exists a unique speed $c(\p)$ such that $p(\v_{c(\p)}) = \p$. Hence, we can express $E(\v_c) \equiv \boE(p(\v_c))$ as a function of $ p(\v_c)$ to obtain the following graph.

\bigskip

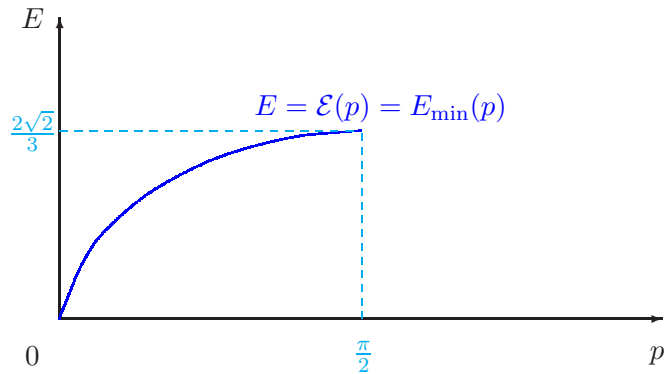
\begin{center}
\begin{picture}(90,50)(0,0)
\linethickness{0.2mm}
\put(10,10){\line(1,0){80}}
\put(90,10){\vector(1,0){0.12}}
\linethickness{0.2mm}
\put(10,10){\line(0,1){40}}
\put(10,50){\vector(0,1){0.12}}
\linethickness{0.2mm}
{\color{blue}
\qbezier(8.7,10)(8.77,10.35)(10.25,13.92)
\qbezier(10.25,13.92)(11.72,17.78)(13.7,20.5)
\qbezier(13.7,20.5)(16.8,24)(20.61,26.7)
\qbezier(20.61,26.7)(24.42,29.16)(28.7,31)
\qbezier(28.7,31)(34.21,33.2)(41.14,34.4)
\qbezier(41.14,34.4)(48.06,34.98)(48.7,35)
}
\linethickness{0.1mm}
{\color{cyan}
\multiput(46.3,10)(0,1.82){14}{\line(0,1){0.91}}
}
\linethickness{0.1mm}
{\color{cyan}
\multiput(3.6,35)(2,0){20}{\line(1,0){1}}
}
\put(0,5){\makebox(0,0)[cc]{$0$}}
\put(0,50){\makebox(0,0)[cc]{$E$}}
\put(83,5){\makebox(0,0)[cc]{$p$}}
{\color{cyan}
\put(44,5){\makebox(0,0)[cc]{$\frac{\pi}{2}$}}
\put(0,35){\makebox(0,0)[cc]{$\frac{2 \sqrt{2}}{3}$}}
}
{\color{blue}
\put(45,38){\makebox(0,0)[cc]{$E = \boE(p) = E_{\min}(p)$}}
}
\end{picture}
\end{center}
By Lemma \ref{tildeemin}, the curve $p \mapsto \boE(p)$ is identically equal to the minimizing curve $p \mapsto E_{\min}(p)$. It is a smooth, increasing and strictly concave curve, which lies below the line $E = \sqrt{2} p$. Each point of the curve represents a non-constant solution $\v_c$ to \eqref{TWc} of energy $E(\v_c)$ and scalar momentum $p(\v_c)$. The speed of the solution $\v_c$ (and as a result, its position on the curve) is given by the slope of the curve. Indeed, it follows from \eqref{piquet} and \eqref{monaco} that
$$\frac{{\rm d} \boE}{{\rm d} p} \big( p(\v_c) \big) = \frac{\rm d}{{\rm d} c} \Big( E(\v_c) \Big) \bigg( \frac{\rm d}{\rm dc} p(\v_c) \bigg)^{-1} = c.$$
\end{remark}

\begin{remark}
Since $p(\overline{v}) = - p(v)$ for any function $v \in \tilde{\boX^1}$, it follows from Lemma \ref{tildeemin} that, given any $- \frac{\pi}{2} < \p < 0$, problem \eqref{cochette1} is achieved by a unique minimizer, up to the invariances of the problem, the map $\v_{- c(\p)}$, where $c(\p)$ denotes the unique speed $c$ such that $p(\v_c) = - \p$.
\end{remark}

We emphasize that the definition of the normalized momentum $p$ is restricted to maps having no zeroes, and therefore not to $\v_0$. To define a minimization problem similar to \eqref{cochette1}, we make use of the untwisted momentum $[p]$ and consider the quantity
\begin{equation}
\label{porcelet}
\E_{\min} \Big( \frac{\pi}{2} \Big) \equiv \inf \big\{ E(v), v \in \boX^1 \ {\rm s.t.} \ [p](v) = \frac{\pi}{2} \ {\rm mod} \ \pi \big\}.
\end{equation}
We will prove

\begin{lemma}
\label{verrat}
The infimum $\E_{\min}(\frac{\pi}{2})$ is achieved by the map $\v_0$, which is the only minimizer up to the invariances of problem \eqref{porcelet}.
\end{lemma}

As a matter of fact, the central part of the variational argument entering in the proof of the orbital stability is a careful analysis of sequences $(u_n)_{n \in \N}$ in the space $\boX^1$ verifying
\begin{equation}
\label{truie}
\begin{split}
[p_n] & \equiv [p](u_n) \to \frac{\pi}{2},\\
& {\rm and}\\
E(u_n) & \to \ \E_{\min} \Big( \frac{\pi}{2} \Big), \ {\rm as} \ n \to + \infty.
\end{split}
\end{equation}
Minimizing sequences for $\E_{\min}(\frac{\pi}{2})$ are a special example of sequences satisfying \eqref{truie}. We have

\begin{theorem}
\label{laye}
Let $(u_n)_{n \in \N}$ be a sequence of maps in the space $\boX^1$ satisfying \eqref{truie}. There exist a subsequence $(u_{\sigma(n)})_{n \in \N}$, a sequence of points $(a_n)_{n \in \N}$, and a real number $\theta$ such that
$$u_{\sigma(n)} \big( \cdot + a_{\sigma(n)} \big) \to \exp i \theta \, \v_{0}(\cdot), \ {\rm as} \ n \to + \infty,$$
uniformly on any compact subset of $\R$. Moreover,
$$1-\vert u_{\sigma(n)} \big( \cdot + a_{\sigma(n)} \big) \vert^2 \to 1 - \vert \v_{0}(\cdot) \vert^2 \ {\rm in} \ L^2(\R), \ {\rm as} \ n \to + \infty,$$
and
$$u_{\sigma(n)}' \big( \cdot + a_{\sigma(n)} \big) \to \exp i \theta \, \v_{0}'(\cdot) \ {\rm in} \ L^2(\R), \ {\rm as} \ n \to + \infty.$$
\end{theorem}

Lemma \ref{verrat} is an immediate consequence of Theorem \ref{laye}, taking minimizing sequences for $\E_{\min}(\frac{\pi}{2})$. In the study of the orbital stability of $\v_0$, another important consequence of Theorem \ref{laye} is

\begin{cor}
\label{porcide}
Given any numbers $A > 0$ and $\varepsilon > 0$, there exists a number $\delta > 0$ such that, if the function $v$ belongs to the space $\boX^1$ and satisfies
$$\Big| [p](v) - \frac{\pi}{2} \Big| \leq \delta, \ {\rm and} \ |E(v) - E(\v_0)| \leq \delta,$$
then there exist some numbers $a \in \R$ and $\theta \in \R$ such that
$$d_{A, \boX^1} \big( v(\cdot + a), \exp i \theta \ \v_0(\cdot) \big) \leq \varepsilon.$$
\end{cor}

We next turn to the dynamics of \eqref{GP}. The second part in the proof of Theorem \ref{stationnaire} is to establish that besides the energy, which is known to be conserved in view of Theorem \ref{cauchy}, the untwisted momentum $[p]$ is preserved by \eqref{GP}.

\begin{prop}
\label{grouic}
Assume $v_0 \in \boX^1$, and let $v$ be the solution to \eqref{GP} with initial datum $v_0$. Then,
$$[p](v(\cdot, t)) = [p](v_0), \ \forall t \in \R.$$
If moreover $v_0 \in \boZ^1$, then $v(t)$ belongs to $\boZ^1$ for any $t \in \R$, and
$$\boP(v(\cdot, t)) = \boP(v_0), \ \forall t \in \R.$$
\end{prop}

The proof of Theorem \ref{stationnaire} then follows combining Theorem \ref{laye} with the conservation of energy and untwisted momentum, and the continuity of the latter with respect to $d_{A,\boX^1}$.

Finally, for the proof of Theorem \ref{vitesse}, we invoke the conservation law for the relative center of mass. At least formally, we have the identity
\begin{equation}
\label{interP}
\frac{d}{dt} \bigg( \frac{1}{2} \int_{\R} x \Big( |\Psi(x,t)|^2 - 1 \Big) dx \bigg) = 2 P(\Psi(t)),
\end{equation}
for solutions to \eqref{GP}. The rigorous argument in Section \ref{sect:4} involves a localized version of \eqref{interP}.

The rest of the paper is organized as follows. In the next section, we provide proofs of various results stated in the introduction, and to several properties of maps having a bounded Ginzburg-Landau energy. We prove Theorem \ref{laye} in Section \ref{sect:3}, while Section \ref{sect:4} is devoted to the rigorous proofs of the conservation of the untwisted momentum and center of mass. Finally, we prove Theorem \ref{stationnaire} and Theorem \ref{vitesse} in Section \ref{sect:5}.

\begin{merci}
F.B., P.G. and D.S. acknowledge partial support from project JC05-51279 of the Agence Nationale de la Recherche. J.-C. S. acknowledges support from project ANR-07-BLAN-0250 of the Agence Nationale de la Recherche.
\end{merci}

\section{Properties of energy and momentum}
\label{sect:2}

The purpose of this section is to provide several properties of maps with bounded energy, in particular in connection with their momentum and possible limits at infinity. We also provide the proofs of various results stated in the introduction.

\subsection{Maps with finite Ginzburg-Landau energy}

We first have

\begin{lemma}
\label{changi}
Let $E > 0$ and $0 < \delta_0 < 1$ be given. There exists an integer $\ell_0 = \ell_0(E, \delta_0)$, depending only on $E$ and $\delta_0$, such that the following property holds: given any map $v \in \boX^1$ satisfying $E(v) \leq E$, either
$$\big| 1 - |v(x)| \big| < \delta_0, \ \forall x \in \R,$$
or there exists $\ell$ points $x_1$, $x_2$, $\ldots$, and $x_\ell$ satisfying $\ell \leq \ell_0$,
$$\big| 1 - |v(x_i)| \big| \geq \delta_0, \ \forall 1 \leq i \leq \ell,$$
and
$$\big| 1 - |v(x)| \big| \leq \delta_0, \ \forall x \in \R \setminus \underset{i = 1}{\overset{\ell}{\cup}} \big[ x_i - 1, x_i + 1 \big].$$
\end{lemma}

\begin{proof}
Set
$$\boA = \big\{ z \in \R, \ {\rm s.t.} \ \big| 1 - |v(z)| \big| \geq \delta_0 \big\},$$
and assume that $\boA$ is not empty. Considering the covering $\R = \underset{i \in \N}{\cup} I_i$, where $I_i = [i - \frac{1}{2}, i + \frac{1}{2}]$, we claim that, if $I_i \cap \boA \neq \emptyset$, then
\begin{equation}
\label{claim0}
\int_{\tilde{I_i}} e(v) \geq \mu_0,
\end{equation}
where $\tilde{I_i} = [i - 1, i + 1]$, and $\mu_0$ is some positive constant. To prove the claim, we first notice that
$$|v(x) - v(y)| \leq \| v' \|_{L^2(\R)} |x - y|^\frac{1}{2} \leq \sqrt{2} E^\frac{1}{2} |x - y|^\frac{1}{2},$$
for any $(x, y) \in \R^2$. Therefore, if $z \in \boA$, then
$$\big| 1 - |v(y)| \big| \geq \frac{\delta_0}{2}, \ \forall y \in [z - r, z + r],$$
where $r = \frac{\delta_0^2}{8 E} $. Choosing $r_0 = \min \{ r, \frac{1}{2} \}$, we are led to
$$\int_{z - r_0}^{z + r_0} e(v) \geq\frac{1}{4} \int_{z - r_0}^{z + r_0} (1 - |v|)^2 \geq \mu_0 \equiv \frac{r_0 \delta_0^2}{8}.$$
In particular, if $z \in I_i \cap \boA$ for some $i \in \N$, then $[z - r_0, z + r_0] \subset \tilde{I_i}$, and claim \eqref{claim0} follows. To conclude the proof, we notice that
$$\underset{i \in \N}{\sum} \int_{\tilde{I_i}} e(v) = 2 E(v) \leq 2 E,$$
so that, in view of \eqref{claim0},
$$\ell \mu_0 \leq 2 E,$$
where $\ell = \Card \{ i \in \N, \ {\rm s.t.} \ I_i \cap \boA \neq \emptyset \}$. Setting $\ell_0 = \frac{2 E}{\mu_0}$, and choosing some point $x_i \in I_i \cap \boA$, for any $i \in \N$ such that $I_i \cap \boA \neq \emptyset$, the conclusion follows (after a possible relabelling of the points $x_i$).
\end{proof}

An immediate consequence of Lemma \ref{changi} is

\begin{cor}
\label{longueur} 
Given any map $v \in \boX^1$, and any number $0 < \delta < 1$, there exists some number $L(\delta, v) \geq 0$ such that
$$|v(x)| \geq \delta,$$
for any $|x| \geq L(\delta, v)$. In particular,
\begin{equation}
\label{salami}
|v(x)| \to 1, \ {\rm as} \ |x| \to + \infty.
\end{equation}
\end{cor}

\begin{remark}
We emphasize once more that in contrast with \eqref{salami}, the map $v$ itself need not have a limit as $|x| \to + \infty$. It suffices to choose $v = \exp i \varphi$, with $\varphi'$ belonging to $L^2(\R)$, but not to $L^1(\R)$, for instance $\varphi(x) = \ln( x^2 + 1)$.
\end{remark}

We finish this section with some elementary observations. The first one emphasizes the role of the sonic speed $\sqrt{2}$. 

\begin{lemma}
\label{colisee}
Let $\varrho$ and $\varphi$ be real-valued, smooth functions on some interval of $\R$, such that $\varrho$ is positive. Set $v = \varrho \exp i \varphi$. Then, we have the pointwise bound
$$\Big| (\varrho^2 - 1) \varphi' \Big| \leq \frac{\sqrt{2}}{\varrho} e(v).$$
\end{lemma}

\begin{proof}
The energy density of $v$ can be expressed as
$$e(v) = \frac{1}{2} \Big( (\varrho')^2 + \varrho^2 (\varphi')^2 \Big) + \frac{1}{4} \Big( 1 - \varrho^2 \Big)^2.$$
The conclusion follows from the inequality $|a b| \leq \frac{1}{2} (a^2 + b^2)$ applied to $a = \frac{1}{\sqrt{2}} (1 - \varrho^2)$ and $b = \varrho \varphi'$.
\end{proof}

As a consequence, we have

\begin{cor}
\label{coro}
Assume $v \in \tilde{\boX^1}$. Then,
$$\underset{x \in \R}{\inf} |v(x)| \leq \frac{E(v)}{\sqrt{2} |p(v)|}.$$
In particular, if $\delta(v) \equiv 1 - \frac{E(v)}{\sqrt{2} |p(v)|} > 0$, then, given any $0 < \delta < \delta(v)$, there exists some point $x_\delta \in \R$ such that
$$1 - |v(x_\delta)| \geq \delta.$$
\end{cor}

\begin{proof}
Set $\delta_0 = \underset{x \in \R}{\inf} |v(x)|$. Since $|v(x)| \to 1$, as $|x| \to + \infty$, and $v$ is continuous and does not vanish on $\R$, its infimum $\delta_0$ is positive. Moreover, writing $v = \varrho \exp i \varphi$, it follows from Lemma \ref{colisee} that we have the pointwise bound
$$\Big| (\varrho^2-1) \varphi' \Big| \leq \frac{\sqrt{2}}{\delta_0} e(v),$$
and the conclusion follows from formula \eqref{defnormmo} by integration.
\end{proof}

Notice that in contrast, if $v \in \boX^1\setminus \tilde{\boX^1}$, then $\underset{x\in \R}{\inf} |v(x)| = 0$.

\subsection{Minimality of the kink solution}

The kink solution has the following remarkable minimization property.

\begin{lemma} 
\label{kinkmini}
We have
$$E(\v_0) = \inf \Big\{ E(v), v \in H^1_{\rm loc}(\R), \underset{x \in \R}{\inf} \big| v(x) \big| = 0 \Big\}.$$
In particular, if $E(v) < \frac{2 \sqrt{2}}{3}$, then,
$$\underset{x \in \R}{\inf} \big| v(x) \big| > 0.$$
\end{lemma}

\begin{proof}
We consider a minimizing sequence $(v_n)_{n \in \N}$ for the minimizing problem
$$\boE_0 = \inf \bigg\{ \int_0^{+ \infty} e(v), v \in H^1_{\rm loc}([0, + \infty)), v(0) = 0 \bigg\},$$
which is well-defined by Sobolev embedding theorem. We notice that the functions $v_n'$ are uniformly bounded in $L^2([0, + \infty))$, and that $v_n(0) = 0$. Hence, by Rellich compactness theorem, there exists some function $u \in H^1_{\rm loc}([0, + \infty))$, with $u(0) = 0$, such that, up to a subsequence,
$$v_n' \rightharpoonup u' \ {\rm in} \ L^2([0, + \infty)), \ {\rm and} \ v_n \to u \ {\rm in} \ L^\infty_{\rm loc}([0, + \infty)), \ {\rm as} \ n \to + \infty.$$
By Fatou lemma, we are led to
$$\int_0^{+ \infty} e(u) = \frac{1}{2} \int_0^{+ \infty} {u' }^2 + \frac{1}{4} \int_0^{+ \infty} \underset{n \to + \infty}{\liminf} \big( 1 - |v_n|^2 \big)^2 \leq \underset{n \to + \infty}{\liminf} \int_0^{+ \infty} e(v_n),$$
so that the infimum $\boE_0$ is achieved by the function $u$. In particular, the solution $u$ is critical for the Ginzburg-Landau energy, i.e. it solves
$$u'' + u (1 - |u|^2) = 0.$$
Integrating this equation yields $u(x) = \v_0(x) = \th \Big( \frac{x}{\sqrt{2}} \Big)$, so that
$$\boE_0 = \int_0^{+ \infty} e(\v_0) = \frac{\sqrt{2}}{3}.$$
Next consider a map $v \in H^1_{\rm loc}(\R)$, with finite Ginzburg-Landau energy, and which vanishes at some point $x_0$. In view of the invariance by translation, we may assume that $x_0 = 0$, whereas the minimality of $\v_0$ yields
$\int_0^{+ \infty} e(v) \geq \frac{\sqrt{2}}{3}$,
and the same inequality holds for the energy on $(- \infty, 0]$, so that
$$E(v) \geq \frac{2 \sqrt{2}}{3} = E(\v_0).$$
The proof of Lemma \ref{kinkmini} follows.
\end{proof}

\subsection{Properties of the momentum} 

We provide in this subsection the proofs of Lemmas \ref{limite}, \ref{andouille}, \ref{limite2} and \ref{lem:andouille2}, as well as some additional properties.

\begin{proof}[Proof of Lemma \ref{limite}]
Let $v \in \boX^1$, and let $L = L(\frac{1}{2}, v)$ be the corresponding number provided by Corollary \ref{longueur} for $\delta = \frac{1}{2}$. On the intervals $(\pm L, \pm \infty)$, we may write $v = \varrho \exp i \varphi_\pm$, so that $\langle i v, v' \rangle = \varrho^2 \varphi_\pm'$. Next, given $R_2 > R_1 > L$, we have
\begin{equation}
\label{cauchyR}
P_{R_2}(v) - P_{R_1}(v) = \frac{1}{2} \int_{R_1}^{R_2} \varrho^2 \varphi_{+}' + \frac{1}{2} \int_{- R_2}^{- R_1} \varrho^2\varphi_{-}'. 
\end{equation}
We expand
$$\int_{\pm R_1}^{\pm R_2} \varrho^2 \varphi_{\pm}' = \int_{\pm R_1}^{\pm R_2} (\varrho^2 - 1) \varphi_{\pm}' + \int_{\pm R_1}^{\pm R_2} \varphi_{\pm}'.$$
For the first integral on the right-hand side, we use the bound $|(\varrho^2 - 1) \varphi_{\pm}'| \leq 4e(v)$ on $(\pm L, \pm\infty)$, so that 
 $$\bigg| \int_{\pm R_1}^{\pm R_2} (\varrho^2 - 1) \varphi_{\pm}' \bigg| \leq 
4 \int_{\pm R_1}^{\pm R_2} e(v) \to 0, \ {\rm as} \ R_1 \to + \infty .$$
For the second integral, since $v \in \boZ^1$, it has limits at infinity, so that
\begin{equation*}
\bigg| \int_{\pm R_1}^{\pm R_2} \varphi_{\pm}' \bigg| = \Big| \varphi_\pm(R_2) - \varphi_\pm(R_1) \Big| \to 0, \ {\rm as} \ R_1 \to + \infty. 
\end{equation*}
Hence,
$$P_{R_2}(v) - P_{R_1}(v) \to 0, \ {\rm as} \ R_1 \to + \infty,$$
and therefore $P_R(v)$ has a limit, which establishes the first statement.

Concerning the second statement, if $v$ belongs to $\tilde{\boZ^1}$, then $v=\varrho\exp i \varphi$ on the whole space $\R$, so that $\langle i v, v' \rangle = \varrho^2 \varphi'$, and
$$P_R(v) = \frac{1}{2} \int_{- R}^R \varrho^2 \varphi' = \frac{1}{2} \int_{- R}^R (\varrho^2 - 1) \varphi' + \frac{1}{2} \int_{- R}^R \varphi',$$
for any $R > 0$. The conclusion then follows as above.
\end{proof}

\begin{proof}[Proof of Lemma \ref{andouille}]
Let $V_0 \in \boZ^1$ and $w \in H^1(\R)$. Since
\begin{equation}
\label{guemene}
w(x) \to 0, \ {\rm as} \ |x| \to + \infty,
\end{equation}
and since $V_0$ admits limits at infinity, so does $V_0 + w$. 

Next, expanding
\begin{align*}
(1 - |V_0 + w|^2)^2 = (1 - & |V_0|^2)^2 + |w|^4 + 4 (\langle V_0, w \rangle)^2\\
- & 4 (1 - |V_0|^2 - |w|^2)(\langle V_0, w \rangle) - 2 (1 - |V_0|^2) |w|^2,
\end{align*}
and using the fact that $V_0$ is bounded in $L^\infty(\R)$, as well as the fact that $w$ and $1 - |V_0|^2$ are bounded in $L^q(\R)$ for any $2 \leq q \leq +\infty$, one checks that
$$E(w + V_0) < + \infty,$$
and therefore $V_0 + w \in \boZ^1$. Concerning formula \eqref{andouille1}, we have
$$\langle i (V_0 + w), (V_0 + w)' \rangle = \langle i V_0, V_0' \rangle + \langle i w, w' \rangle + \langle i w, V_0' \rangle + \langle i V_0, w' \rangle,$$
so that
$$\boP_R(w + V_0) = \boP_R(V_0) + \boP_R(w) + \frac{1}{2} \int_{- R}^R \big( \langle i w, V_0' \rangle + \langle i V_0, w' \rangle \big),$$
for any $R > 0$. Integrating by parts, we obtain
$$\int_{- R}^R \langle i V_0, w' \rangle = - \int_{- R}^R \langle i V_0', w \rangle + \big[\langle i V_0, w \rangle \big]_{- R}^R = \int_{- R}^R \langle i w, V_0' \rangle + \big[ \langle i V_0, w \rangle \big]_{- R}^R,$$
where the last term tends to $0$, as $R \to +\infty$ by \eqref{guemene}. Combining the previous identities, we are led to
\begin{equation}
\label{andouille4}
\boP_R(w + V_0) = \boP_R(V_0) + \boP_R(w) + \int_{- R}^R \langle i w, V_0' \rangle + \underset{R \to +\infty}{o}(1),
\end{equation}
and \eqref{andouille1} follows letting $R \to + \infty$.
\end{proof}

\begin{proof}[Proof of Lemma \ref{limite2}]
It is very similar to the proof of Lemma \ref{limite}. It suffices to replace \eqref{cauchyR} by an equality in $\R/\pi \Z$,
\begin{equation}
\label{cauchyR'}
\begin{split}
& \Big( P_{R_2}(v) - \frac{1}{2} \big( \arg v(R_2) - \arg v(- R_2) \big) \Big) - \Big( P_{R_1}(v)- \frac{1}{2} \big( \arg v(R_1) - \arg v(- R_1) \big) \Big)\\
& = \frac{1}{2} \int_{R_1}^{R_2} \varrho^2 \varphi_{+}' - \frac{1}{2} \big( \varphi_+(R_2) - \varphi_+(R_1) \big) + \frac{1}{2} \int_{- R_2}^{- R_1} \varrho^2 \varphi_{-}' - \frac{1}{2} \big( \varphi_-(- R_1) - \varphi_-(- R_2) \big)\\
& = \frac{1}{2} \int_{R_1}^{R_2} (\varrho^2 - 1) \varphi_{+}' + \frac{1}{2} \int_{- R_2}^{- R_1} (\varrho^2 - 1) \varphi_{-}',
\end{split}
\end{equation}
and the remaining part of the proof is identical.
\end{proof} 

\begin{proof}[Proof of Lemma \ref{lem:andouille2}]
Again, it is almost identical to the proof of Lemma \ref{andouille}, the conclusion being obtained by subtracting one half of $\arg(w + V_0)(R) - \arg(w + V_0)(- R)$ from both sides of \eqref{andouille4} and noticing that in view of \eqref{guemene}, $\arg(w + V_0)(\pm R) - \arg V_0(\pm R)$ tends to $0$, as $R \to
+ \infty$.
\end{proof}

We will also use the continuity of the untwisted momentum.

\begin{lemma}
\label{lem:lip}
The untwisted momentum $[p]$ is a locally Lipschitz continuous map from $\boX^1$ to $\R / \pi \Z$ for the distance $d_{A, \boX^1}$. 
\end{lemma}

\begin{proof}
Let $u \in \boX^1$ be given and $R>0$ be such that $|u(x)| \geq \frac{1}{2}$ on $\R \setminus (- R, R)$. If $\delta > 0$ is sufficiently small, and if $v\in \boX^1$ is such that $d_{A, \boX^1}(u, v) \leq \delta$, then $|v(x)| \geq \frac{1}{3}$ on $\R \setminus (- R, R)$. We then have, in view of \eqref{cauchyR'} and the definition of $[p]$,
\begin{equation}
\label{eq:alette}
\begin{split}
[p](u) - [p](v) & = \big( P_R(u) - P_R(v) \big) - \frac{1}{2} \big( \arg u(R) - \arg v(R) \big) + \frac{1}{2} \big( \arg u(- R) - \arg v(- R) \big)\\
& + \frac{1}{2} \int_{- \infty}^{- R} \Big( (|u|^2 - |v|^2) \varphi_{u, -}' +
(|v|^2 - 1) (\varphi_{u, -}' - \varphi_{v, -}') \Big)\\
& + \frac{1}{2} \int_R^{+ \infty} \Big( (|u|^2 - |v|^2) \varphi_{u, +}' +
(|v|^2 - 1) (\varphi_{u, +}' - \varphi_{v, +}') \Big) \ {\rm mod} \ \pi,
\end{split}
\end{equation} 
where $\varphi_{u, \pm}$ and $\varphi_{v, \pm}$ denote representatives for the phases of $u$ and $v$ on $(- \infty, - R]$ and $[R, + \infty)$. By Cauchy-Schwarz inequality, the terms in the last two lines of \eqref{eq:alette} are bounded by a constant (which depends only on $E(u)$) times $d_{A,\boX^1}(u,v)$. By Cauchy-Schwarz inequality and Sobolev embedding theorem, the terms in the first line of \eqref{eq:alette} are bounded by a constant (which depends on $R$ and $E(u)$) times $d_{A,\boX^1}(u,v)$. This completes the proof of Lemma \ref{lem:lip}.
\end{proof}

\begin{remark}
Refining the proof of Lemma \ref{lem:lip} by using the decomposition in Lemma \ref{changi}, one could actually show that the local Lipschitz constant depends only on a bound on $E(u)$, and not on $R$.
\end{remark}

In our study of the minimization problem $\E_{\min}(\frac{\pi}{2})$, we shall need the following construction.

\begin{lemma}
\label{construct}
Let $0 < |\q| \leq \frac{1}{32}$ and $0 \leq \mu \leq \frac{1}{4}$. There exists some number $\ell > 1$, and a map $w = |w| \exp i \psi \in H^1([0, \ell])$, such that
\begin{equation}
\label{quoidonc0}
w(0) = w(\ell), \ \big| 1- |w(0)| \big \vert = \mu,
\end{equation}
\begin{equation}
\label{quoidonc1}
\q = \frac{1}{2} \int_0^\ell |w|^2 \psi',
\end{equation}
and
\begin{equation}
\label{quoidonc3}
E(w) \leq 14 |\q|.
\end{equation}
\end{lemma}

\begin{proof}
Consider the functions $f_1$ and $\psi_1$ defined on the interval $[0, 2]$ by
$$f_1(s) = s \ {\rm on } \ \Big[ 0, \frac{1}{2} \Big], \ f_1(s) = 1 - s \ {\rm on } \ \Big[ \frac{1}{2}, 1 \Big], \ {\rm and} \ f_1(s) = 0 \ {\rm on} \ \Big[ 1, 2 \Big],$$
and
$$\psi_1(s) = s \ {\rm on} \ [0, 1], \ {\rm and} \ \psi_1(s) = 2 - s \ {\rm on} \ [0, 1].$$
For a given positive number $\lambda$, we consider the functions defined on $[0, 2 \lambda]$ by
$$f_\lambda(s) = \frac{1}{\lambda} f \Big( \frac{s}{\lambda} \Big), \ {\rm and} \ \psi_\lambda(s) = \psi \Big( \frac{s}{\lambda} \Big),$$
so that $|f_\lambda| \leq \frac{1}{2 \lambda}$, $|\psi_\lambda'| = \frac{1}{\lambda}$, $f_\lambda(0) = f_\lambda(2 \lambda) = 0$, $\psi_\lambda(0) = \psi_\lambda(2 \lambda) = 0$, and
\begin{equation}
\label{vroum-vroum}
\int_0^{2 \lambda} f_\lambda \psi_\lambda' = \frac{1}{4 \lambda}, \ \int_0^{2 \lambda} f_\lambda = \frac{1}{4}, \ \int_0^{2 \lambda} f_\lambda^2 = \frac{1}{12 \lambda}, \int_0^{2 \lambda} (f_\lambda')^2 = \frac{1}{\lambda^3}, \ {\rm and} \ \int_0^{2 \lambda} (\psi_\lambda')^2 = \frac{2}{\lambda}.
\end{equation}
We then choose $\lambda = \frac{1}{8 |\q|}$, so that, $\frac{1}{\lambda} \leq \frac{1}{4}$, introduce a new parameter $\delta > 0$ to be determined later, and consider the function
$$\rho_{\lambda, \delta} = \sqrt{1 - \delta - f_\lambda},$$
so that $1 - \rho_{\lambda, \delta}^2 = f_\lambda + \delta$. It follows from our choice of parameter $\lambda$ that
\begin{equation}
\label{revroum}
|\q| = \frac{1}{2} \int_0^{2 \lambda} f_\lambda \psi_\lambda' = \frac{1}{2} \int_0^{2 \lambda} (f_\lambda + \delta - 1) \psi_\lambda' = - \frac{1}{2} \int_0^{2 \lambda} \rho_{\lambda, \delta}^2 \psi_\lambda'.
\end{equation}
We finally choose $\ell = 2\lambda$ and
$$w = \left\{ \begin{array}{ll} \rho_{\lambda, \delta} \exp (- i \psi_\lambda), \ {\rm if} \ \q > 0,\\ \rho_{\lambda, \delta} \exp i \psi_\lambda, \ {\rm if} \ \q < 0. \end{array} \right.$$
Condition \eqref{quoidonc1} is fulfilled with this choice of $w$ in view of \eqref{revroum}. Moreover, by construction, $w(0) = w(\ell) = \sqrt{1 - \delta}$, so that conditions \eqref{quoidonc0} are satisfied for any $\delta \leq \mu^2$. We finally compute
$$E(w) = \int_0^{2 \lambda} \bigg( \frac{(f_\lambda')^2}{8 (1 - \delta - f_\lambda)} + \Big( 1 - \delta - f_\lambda \Big) \frac{(\psi_\lambda')^2}{2} + \frac{f_\lambda^2}{4} + \frac{\delta f_\lambda}{2} + \frac{\delta^2}{4} \bigg),$$
so that, since
$$0 \leq f_\lambda + \delta \leq \frac{1}{2 \lambda} + \delta \leq \mu^2 + \frac{1}{8} \leq \frac{1}{2},$$
it follows from \eqref{vroum-vroum} that
$$E(w) \leq \int_0^{2 \lambda} \bigg( \frac{(f_\lambda')^2}{4} + \frac{(\psi_\lambda')^2}{2} + \frac{f_\lambda^2}{4} + \frac{\delta f_\lambda}{2} + \frac{\delta^2}{4} \bigg) \leq \frac{1}{4 \lambda^3} + \frac{49}{48 \lambda} + \frac{\delta}{8} + \frac{\delta^2 \lambda}{2}.$$
Inequality \eqref{quoidonc3} follows choosing $\delta = \min \{ \mu^2, \frac{1}{\lambda} \}$. 
\end{proof}

\section{Proof of Theorem \ref{laye}}
\label{sect:3}

In this section, we undertake the study of sequences $(u_n)_{n \in \N}$ in the space $\boX^1$ verifying
\begin{equation}
\label{truie2}
\begin{split}
[p_n] & \equiv [p](u_n) \to \frac{\pi}{2},\\
& {\rm and}\\
E(u_n) & \to \E_{\min} \Big( \frac{\pi}{2} \Big), \ {\rm as} \ n \to + \infty.
\end{split}
\end{equation}
This study will eventually lead us to the proof of Theorem \ref{laye}. We first have

\begin{lemma}
\label{onlyweak}
Let $(u_n)_{n \in \N}$ be a sequence of maps satisfying \eqref{truie2}. Then, there exist a subsequence $(u_{\sigma(n)})_{n \in \N}$ and a solution $v_c$ to \eqref{TWc} such that
$$u_{\sigma(n)} \rightharpoonup v_c \ {\rm in} \ H^1([- B, B]), \ {\rm as} \ n \to + \infty,$$
for any $B > 0$ Moreover, when $v_c$ is not identically constant, there exist some numbers $\theta$ and $a$ such that $v_c = \exp i \theta \ \v_c(\cdot + a)$.
\end{lemma}

\begin{proof}
Since $(E(u_n))_{n \in \N}$ is bounded by assumption \eqref{truie2}, it follows from standard compactness results that there exists a subsequence $(u_{\sigma(n)})_{n \in \N}$, and a map $u \in H^1_{\rm loc}(\R)$ such that
$$u_{\sigma(n)} \rightharpoonup u \ {\rm in} \ H^1([- B, B]), \ {\rm as} \ n \to + \infty,$$
for any $B > 0$. It remains to prove that the limiting map $u$ solves \eqref {TWc} on $(- B, B)$. For that purpose, we consider a smooth map $\xi$, with compact support in $(-B, B)$, such that
\begin{equation}
\label{ortho}
\int_{\R} \langle i u, \xi' \rangle = 0.
\end{equation}
We claim that, for any $t$ sufficiently small,
\begin{equation}
\label{claim}
\int_{- B}^B e(u_{\sigma(n)} + t \xi) \geq \int_{- B}^B e(u_{\sigma(n)}) + O(t^2) + \underset{n \to + \infty}{o(1)}.
\end{equation}
To establish the claim, we first expand the momentum $[p](u_n + t \xi)$, using formula \eqref{eq:andouille2} of Lemma \ref{lem:andouille2}. We obtain in $\R / \pi \Z$, 
\begin{align*}
[p](u_n + t \xi) & = [p](u_n) + t \int_{\R} \langle i u_n, \xi' \rangle + O(t^2) \\
& = [p_n] + O(t^2) + \underset{n \to + \infty}{o(1)} = \frac{\pi}{2} + O(t^2) + \underset{n \to + \infty}{o(1)},
\end{align*}
so that, setting $\q_{n, t} = \frac{\pi}{2} - [p](u_n + t \xi)$, we are led to
$$\q_{n, t} = O(t^2) + \underset{n \to + \infty}{o(1)}.$$
We next construct a comparison map $v_{n, t}$ for $\E_{\min}(\frac{\pi}{2})$ applying several modifications to the map $u_n + t \xi$. For that purpose, we invoke Lemma \ref{construct} with $\q = \q_{n, t}$, and $\mu = \mu_{n, t} = \inf \{ \frac{1}{4}, \frac{1}{2}\nu_{n, t} \}$, where $\mu_{n, t} = \sup \{ |1 - |u_n(x)||, x \in [-B, B] \}$. This yields a positive number $\ell_{n, t} > 1$, and a map $w_{n, t} = |w_n| \exp i \psi_n$, defined on $[0, \ell_n(t)]$ such that
$$w_n(0) = w_n(\ell_{n, t}), \ {\rm and} \ \big| 1 - |w_{n, t}(0)| \big| = \mu_{n, t},$$
and such that
$$\q_{n, t} = \frac{1}{2} \int_0^{\ell_{n, t}} |w_{n, t}|^2 \psi_n',$$
and
\begin{equation}
\label{diplodocus}
E(w_{n, t}) \leq 14 |\q_{n ,t}| = O(t^2) + \underset{n \to + \infty}{o(1)}.
\end{equation}
In view of the mean value theorem, there exists some point $x_n$ in $[B, + \infty)$ such that $|u_n(x_n)| = |w_{n, t}(0)|$. Multiplying possibly $w_{n, t}$ by some constant of modulus one, we may therefore assume, without loss of generality, that
$u_n(x_n) = w_n(0)$. We define the comparison map $v_{n, t}$ as follows
\begin{align*}
v_{n, t}(x) & = u_n(x) + t \xi(x), \ \forall x < x_n,\\
v_{n, t}(x) & = w_n(x - x_n), \ \forall x_n \leq x \leq x_n + \ell_{n, t},\\
v_{n, t}(x) & = u_n(x - \ell_{n, t}) + t \xi(x - \ell_{n,t}), \ \forall x \geq x_n + \ell_{n, t}.
\end{align*}
We verify that $v_{n, t}$ belongs to $\boX^1$ and that
$$E(v_{n, t}) = E(u_n + t \xi) + E(w_{n, t}), \ {\rm and} \ [p](v_{n, t}) = [p](u_n + t \xi) + \q_{n, t} = \frac{\pi}{2} \ {\rm mod} \ \pi,$$
so that $v_{n, t}$ is a comparison map for $\E_{\min}(\frac{\pi}{2})$. Therefore we have
\begin{equation}
\label{compa}
E(v_{n, t}) \geq \E_{\min} \Big( \frac{\pi}{2} \Big).
\end{equation}
On the other hand, we have in view of assumption \eqref{truie2},
\begin{equation}
\label{brachiosaurus}
E(u_n) = \E_{\min} \Big( \frac{\pi}{2} \Big) + \underset{n \to + \infty}{o(1)},
\end{equation}
whereas, since $\xi$ has compact support in $(- B, B)$,
\begin{equation}
\label{albertosaurus}
E(u_n + t \xi) - E(u_n) = \int_{- B}^B \Big( e(u_n + t \xi) - e(u_n)\Big).
\end{equation}
Combining \eqref{albertosaurus} with \eqref{brachiosaurus}, \eqref{compa} and \eqref{diplodocus}, we establish claim \eqref{claim}.

To complete the proof of Lemma \ref{onlyweak}, we expand the integral in \eqref{claim} so that
$$t \int_{-B}^B \Big( u_n' \xi' - \xi u_n (1 - |u_n|^2) \Big) \geq O(t^2) + \underset{n \to + \infty}{o(1)}.$$
We then let $n$ tend to $+ \infty$. This yields, in view of the compact embedding of $H^1([-B, B])$ in $C^0([-B, B])$,
$$t \int_{- B}^B \Big( u' \xi'- \xi u(1 - |u|^2) \Big) \geq O(t^2).$$
Letting $t$ tend to $0^+$ and $0^-$, we deduce
$$\int_{- B}^B \Big( u' \xi' - \xi u (1 - |u|^2) \Big) = 0, $$
that is integrating by parts,
$$\int_{\R} \Big( u'' + u (1 - |u|^2) \Big) \xi = 0.$$ 
Since $\xi$ is any arbitrary function with compact support verifying \eqref{ortho}, this shows that there exists some constant $c$ such that $u$ solves \eqref{TWc}. Since any non-constant solution to \eqref{TWc} is of the form $u = \exp i \theta \ \v_c(\cdot + \tilde{x})$ for some $- \sqrt{2} < c < \sqrt{2}$, this yields the conclusion.
\end{proof}

\begin{remark}
Notice that, in the context of Lemma \ref{onlyweak}, we have, as a consequence of lower-semicontinuity and compact embedding theorems,
\begin{equation}
\label{semis}
\int_{- B}^B e(v_c) \leq \underset {n \to + \infty}{\liminf} \int_{- B}^B e(u_{\sigma(n)} ).
\end{equation}
Moreover if $v_c$ has no zero on $[-B, B]$, then this is also the case for $u_{\sigma(n)}$, at least for $n$ sufficiently large, and we may therefore write on $[-B, B]$, $u_{\sigma(n)} = \varrho_{\sigma(n)} \exp i \varphi_{\sigma(n)}$, and $v_c = \varrho_c \exp i \varphi_c$. We then have
\begin{equation}
\label{semis2}
\int_{- B}^B (\varrho_c^2 - 1) \varphi_c' = \underset{n \to + \infty}{\lim} \int_{- B}^B (\varrho_{\sigma(n)}^2 - 1) \varphi_{\sigma(n)}'.
\end{equation}
\end{remark}

It might happen that the limit map provided by Lemma \ref{onlyweak} is a constant of modulus one. To capture the possible losses at infinity, we need to implement a concentration-compactness argument. Assuming first that there exists some positive constant $\delta_0$ such that
$$\underset{x \in \R}{\inf} |u_n(x)| \geq \delta_0,$$
for any $n \in \N$, we are led to
$$|p(u_n)| \leq \frac{E(u_n)}{\sqrt{2} \delta_0},$$
in view of Corollary \ref{coro}. Hence, it follows from \eqref{bacon2} and \eqref{truie2} that, up to some subsequence, there exists some integer $\tilde{k}$ such that
\begin{equation}
\label{plimite}
p(u_n) \to \frac{\pi}{2} + \tilde{k} \pi, \ {\rm as} \ n \to + \infty.
\end{equation}
Setting
$$\delta(u_n) = 1 - \frac{E(u_n)}{\sqrt{2} |p(u_n)|},$$
we deduce from \eqref{truie2} that
$$\delta(u_n) \to \delta_{\frac{\pi}{2}} \equiv 1 - \frac{\E_{\min}(\frac{\pi}{2})}{\sqrt{2} \big| \frac{\pi}{2} + \tilde{k} \pi \big|} \geq 1 - \frac{\sqrt{2} E(\v_0)}{\pi} = 1 - \frac{4}{3\pi} > 0,$$
as $n \to +\infty$. Invoking Lemma \ref{changi}, and Corollary \ref{coro}, we may assert

\begin{prop}
\label{carbonifere}
Let $(u_n)_{n \in \N}$ be a sequence of maps satisfying \eqref{truie2}. There exists an integer $\ell$, depending only on $\delta_{\frac{\pi}{2}}$, and there exist $\ell_n$ points $x_1^n$, $\ldots$, $x_{\ell_n}^n$ satisfying $\ell_n \leq \ell$, such that
$$\big| 1 - |u_n(x_i^n)| \big| \geq \frac{\delta_{\frac{\pi}{2}}}{4}, \ \forall 1 \leq i \leq \ell_n,$$
and
$$\big| 1 - |u_n(x)| \big| \leq \frac{\delta_{\frac{\pi}{2}}}{4}, \ \forall x \in \R \setminus \underset{i = 1}{\overset{\ell_n}{\cup}} \big[ x_i^n - 1, x_i ^n + 1 \big],$$
provided $n$ is sufficiently large. 
\end{prop}

Passing possibly to a further subsequence, we may assume that the number $\ell_n$ does not depend on $n$, and set $\ell = \ell_n$. A standard compactness argument shows that, passing again possibly to another subsequence, and relabelling possibly the points $x_i^n$, we may find some integer $1 \leq \tilde{\ell} \leq \ell$, and some number $R > 0$ such that
\begin{equation}
\label{trias}
|x_i^n - x_j^n| \to + \infty, \ {\rm as} \ n \to + \infty, \ \forall 1 \leq i \neq j \leq \tilde{\ell},
\end{equation}
and
$$x_j^n \in \underset{i = 1}{\overset{\tilde{\ell}}{\cup}} (x_i^n - R, x_i^n + R), \ \forall \tilde{\ell} < j \leq \ell.$$
Going back to Proposition \ref{carbonifere}, we deduce
\begin{equation}
\label{megalodon2}
\big| 1 - |u_n(x)| \big| \leq \frac{\delta_{\frac{\pi}{2}}}{4}, \ \forall x \in \R \setminus \underset{i = 1}{\overset{\tilde{\ell}}{\cup}} (x_i^n - R - 1, x_i^n + R + 1),
\end{equation}
so that, invoking Lemma \ref{colisee}, we have on $\R \setminus \underset{i = 1}{\overset{\tilde{\ell}}{\cup}} (x_i^n - R - 1, x_i^n + R + 1)$,
\begin{equation}
\label{megalodon3}
\frac{1}{2} \Big| (|u_n| ^2 - 1) \varphi_n' \Big| \leq \frac{e(u_n)}{\sqrt{2} \Big( 1 - \frac{\delta_\frac{\pi}{2}}{4} \Big)}.
\end{equation}

We are now in position to provide the proof to Theorem \ref{laye}.

\begin{proof}[Proof of Theorem \ref{laye}]
Since the aim of Theorem \ref{laye} is to provide a subsequence, we may extract subsequences as many times we wish. In our notation, we will not distinguish the subsequence from the original one and will still denote it $(u_n)_{n \in \N}$. We claim

\begin{claim}
\label{grasdouble}
There exists a subsequence $(u_n)_{n \in \N}$ such that $\underset{x \in \R}{\inf} |u_n(x)|$ tends to $0$, as $n \to +\infty$, that is
$$\exists a_n \in \R \ {\rm s.t.} \ u_n(a_n) \to 0, \ {\rm as} \ n\to + \infty.$$
\end{claim}

In this case, we apply Lemma \ref{onlyweak} to the sequence $u_n(\cdot + a_n)_{n \in \N}$ . This shows, that given any arbitrary number $B > A$, there exists a subsequence $(u_n)_{n \in \N}$ and a solution $v_c$ to \eqref{TWc} such that
\begin{equation}
\label{conv10}
u_n(\cdot + a_n) \rightharpoonup v_c \ {\rm in} \ H^1([- B, B]), \ {\rm as} \ n \to + \infty.
\end{equation}
In particular, by compact embedding theorem, the convergence is uniform on the interval $[-B, B]$. It follows therefore from Claim \ref{grasdouble} that
$$v_c(0) = 0.$$
Hence, $c = 0$, since $\v_0$ is the only travelling wave which has a vanishing point, and there exist some number $\theta$ such that $v_c = \exp i \theta \ \v_0$. We have therefore, in view of \eqref{semis},
 \begin{equation}
\label{coppa}
\int_{- B}^B e(\v_0) \leq \underset {n \to + \infty}{\liminf} \int_{- B + a_n}^{B + a_n} e(u_{n} )\leq E(\v_0),
\end{equation}
for any $B > 0$. Given any small $\varepsilon>0$, we choose $B = B_\varepsilon$ so that
$$\int_{|x| > B_\varepsilon} e(\v_0) \leq \frac{\varepsilon}{2},$$
and $n_\varepsilon$ such that
$$E(u_n) \leq E(\v_0) + \frac{\varepsilon}{2},$$
for any $n \geq n_\varepsilon$. Combining with \eqref{coppa}, we obtain
$$\int_{|x - a_n| > B_\varepsilon} e(u_n) \leq \varepsilon,$$
for $n \geq n_\varepsilon$, that is
$$\int_{|x - a_n| > B_\varepsilon} \Big( (u_n')^2+ (1 - |u_n|^2)^2 \Big) \leq 4\varepsilon.$$
Combining with \eqref{conv10}, the conclusion follows.
\end{proof}

\begin{proof}[Proof of Claim \ref{grasdouble}]
Assume by contradiction that there exists $\delta_0 > 0$, such that up to a subsequence $(u_n)_{n \in \N}$, we have
$$\inf_{x \in \R} |u_n(x)| \geq \delta_0,$$
for any $n \in \N$. In this case, we may write $u_n = \varrho_n \exp i \varphi_n$, and we may apply Proposition \ref{carbonifere} to the sequence $(u_n)_{n \in \N}$, so that we may assume that it satisfies \eqref{trias}, \eqref{megalodon2} and \eqref{megalodon3}.

We then divide the proof into several steps.

\begin{step}
\label{stepun}
Given any $1 \leq i \leq \tilde{\ell}$, there exists some numbers $c_i \in (- \sqrt{2}, \sqrt{2}) \setminus \{ 0 \}$, $\tilde{x_i}$ and $\theta_i$ such that
$$u_n(\cdot + x_i^n) \rightharpoonup \exp i \theta_i \ \v_{c_i} (\cdot + \tilde{x_i}) \ {\rm in} \ H^1_{\rm loc}(\R), \ {\rm as} \ n \to + \infty.$$
\end{step}

Applying Lemma \ref{onlyweak} to the sequence $u_n(\cdot + x_i^n)_{n \in \N}$ yields the existence of the limiting solution $v_{c_i}$ to \eqref{TWc}. It remains to prove that the function $v_{c_i}$ is neither a constant function, nor the
kink $\v_0$. This is a consequence of the fact that
$$|u_n(x_i^n)| \leq 1 - \frac{\delta_\frac{\pi}{2}}{4}, \ {\rm and} \ |u_n(x)|\geq \delta_0, \ \forall x \in \R,$$
so that, since we have uniform convergence on compact sets, we obtain
$$|v_{c_i}(0)| \leq 1 - \frac{\delta_\frac{\pi}{2}}{4}, \ {\rm and} \ |v_{c_i}(x)|\geq \delta_0, \ \forall x \in \R.$$

\begin{step}
\label{stepdeux}
Given any number $\mu > 0$, there exist a number $A_\mu > 0$, and $n_\mu \in \N$, such that, if $n \geq n_\mu$, then
$$\int_{\underset{i = 1}{\overset{\tilde{\ell}}{\cup}} (x_i^n - A_\mu, x_i^n + A_\mu)} e(u_n) \geq \underset{i = 1}{\overset{\tilde{\ell}}{\sum}} E(\v_{c_i}) - \mu,$$
and
$$\bigg| \frac{1}{2} \int_{\underset{i = 1}{\overset{\tilde{\ell}}{\cup}} (x_i^n - A_\mu, x_i^n + A_\mu)} (\varrho_n^2 - 1) \varphi_n' - \underset{i = 1}{\overset{\tilde{\ell}}{\sum}} \p_i \bigg| \leq \mu,$$
where $\p_i = p(\v_{c_i})$.
\end{step}

To prove Step \ref{stepdeux}, we choose $A_\mu > R + 1$ so that, for any $1 \leq i \leq \tilde{\ell}$, we have
$$\int_{- A_\mu}^{A_\mu} e(\v_{c_i}) \geq E(\v_{c_i}) - \frac{\mu}{2 \tilde{\ell}},$$
and
$$\frac{1}{2} \bigg| \int_{- A_\mu}^{A_\mu} \Big( (|\v_{c_i}|^2 - 1) \varphi_{c_i}' \Big) - \p_i \bigg| \leq \frac{\mu}{2 \tilde{\ell}}.$$
The conclusion follows from the convergences stated in \eqref{semis} and \eqref{semis2}.

\begin{step}
\label{steptrois}
We have
$$\bigg| \frac{1}{2} \int_{\R \setminus \underset{i = 1}{\overset{\tilde{\ell}}{\cup}} (x_i^n - A_\mu, x_i^n + A_\mu)} (\varrho_n^2 - 1) \varphi_n' \bigg| \leq \frac{1}{\sqrt{2} \Big( 1 - \frac{\delta_\frac{\pi}{2}}{4} \Big)} \int_{\R \setminus \underset{i = 1}{\overset{\tilde{\ell}}{\cup}} (x_i^n - A_\mu, x_i^n + A_\mu)} e(u_n).$$
\end{step}

To establish this inequality, it is sufficient to integrate \eqref{megalodon3}.

Passing possibly to a further subsequence, we may assume that there exist some numbers $\p_\mu$ and $E_\mu$ such that
$$\frac{1}{2} \int_{\R \setminus \underset{i = 1}{\overset{\tilde{\ell}}{\cup}} (x_i^n - A_\mu, x_i^n + A_\mu)} (\varrho_n^2 - 1) \varphi_n' \to \p_{\mu}, \ {\rm and} \ \int_{\R \setminus \underset{i = 1}{\overset{\tilde{\ell}}{\cup}} (x_i^n - A_\mu, x_i^n + A_\mu)} e(u_n) \to E_\mu,$$
as $n \to + \infty$, so that Step \ref{steptrois} yields
$$\sqrt{2} \Big( 1 - \frac{\delta_\frac{\pi}{2}}{4} \Big) |\p_\mu| \leq E_\mu.$$
On the other hand, we have by \eqref{plimite},
$$p(u_n) \to \frac{\pi}{2} + \tilde{k} \pi, \ {\rm as} \ n \to + \infty.$$
Going back to Step \ref{stepdeux}, and letting $n \to + \infty$, we are led to the estimates
$$\bigg| \frac{\pi}{2} + \tilde{k} \pi - \underset{i = 1}{\overset{\tilde{\ell}}{\sum}} \p_i - \p_\mu \bigg| \leq \mu, \ {\rm and } \ \E_{\min} \Big( \frac{\pi}{2} \Big) \geq \underset{i = 1}{\overset{\tilde{\ell}}{\sum}} E_{\min}(\p_i) + E_\mu - \mu.$$
Letting $\mu \to 0$, we may assume that for some subsequence $(\mu_m)_{m \in \N}$ tending to $0$, we have
$$\p_{\mu_m} \to \tilde{\p}, \ {\rm and} \ E_{\mu_m} \to \tilde{E}, \ {\rm as} \ {m \to + \infty}.$$
Our previous inequalities then yield
\begin{equation}
\label{ouida}
\begin{split}
\frac{\pi}{2} + \tilde{k} \pi & = \underset{i = 1}{\overset{\tilde{\ell}}{\sum}} \p_i + \tilde{\p},\\
\E_{\min} \Big( \frac{\pi}{2} \Big) & \geq \underset{i = 1}{\overset{\tilde{\ell}}{\sum}} E_{\min}(\p_i) + \tilde{E},
\end{split}
\end{equation}
with
\begin{equation}
\label{nenni}
\sqrt{2} \Big( 1- \frac{\delta_\frac{\pi}{2}}{4} \Big) |\tilde{\p}| \leq \tilde{E}.
\end{equation}

\begin{step}
\label{stepquatre}
The contradiction.
\end{step}

Notice first that
$$E(\v_0) = \frac{2 \sqrt{2}}{3} \geq \E_{\min} \Big( \frac{\pi}{2} \Big),$$
so that it follows from the strict concavity of the curve $\p \mapsto E_{\min}(\p)$ that
$$E_{\min}(\p_i) > \frac{4\sqrt{2}}{3\pi} |\p_i| \geq \frac{2 \E_{\min}(\frac{\pi}{2})}{\pi} |\p_i|,$$
for any $1 \leq i \leq \tilde{\ell}$. Next, by \eqref{nenni},
$$\tilde{E} \geq \sqrt{2} \Big( 1 - \frac{\delta_\frac{\pi}{2}}{4} \Big) |\tilde{\p}| \geq \sqrt{2} \Big( 1 - \delta_\frac{\pi}{2} \Big) |\tilde \p| \geq \frac{2 \E_{\min}(\frac{\pi}{2})}{\pi} |\tilde{\p}|,$$
where the second inequality is strict unless $\tilde{\p} = 0$. By summation, we therefore obtain in view of \eqref{ouida},
$$\E_{\min} \Big( \frac{\pi}{2} \Big) \geq \underset{i = 1}{\overset{\tilde{\ell}}{\sum}} E_{\min}(\p_i) + \tilde{E} > \frac{2 \E_{\min}(\frac{\pi}{2})}{\pi} \bigg( \sum_{i = 1}^{\tilde{\ell}} |\p_i| + |\tilde{\p}| \bigg) \geq \E_{\min} \Big( \frac{\pi}{2} \Big) \frac{2 |\frac{\pi}{2} + \tilde{k} \pi|}{\pi} \geq \E_{\min} \Big( \frac{\pi}{2} \Big),$$
which yields the desired contradiction.
\end{proof}

\section{Conservation laws for \eqref{GP}}
\label{sect:4}

The purpose of this section is to prove Proposition \ref{grouic} as well as a localised version of \eqref{interP}.

We first recall (see e.g. \cite{Gerard1,Gallo1,Gerard2}) that whenever the initial datum $v_0$ belongs to the space
$$\boX^2 = \Big\{ w \in L^\infty(\R), \ {\rm s.t.} \ w' \in H^1(\R) \ {\rm and } \ 1 - |w|^2 \in L^2(\R) \Big\},$$
the unique global solution $v$ provided by Theorem \ref{cauchy} is in $\boC^0(\R, \boX^2)$, and moreover,
\begin{equation}
\label{eq:regu}
t \mapsto v(t) - v_0 \in \boC^0(\R,H^2(\R)).
\end{equation}

\begin{proof}[Proof of Proposition \ref{grouic}]
Let $v_0 \in \boX^1$ be given, and assume first that $v_0 \in \boX^2$. We write
$$v(t) = v_0 + w(t),$$
so that $w \in \boC^0(\R, H^2(\R))$ by \eqref{eq:regu}, and satisfies
\begin{equation}
\label{eq:perturb}
i w_t + w_{xx} = - {v_0}_{xx} + (w + v_0) (|w + v_0|^2 - 1).
\end{equation}
By Lemma \ref{lem:andouille2}, we have
\begin{equation}
\label{eq:morteau}
[p](v(t)) = [p](v_0) + \frac{1}{2}\int_{\R} \langle i w(t), w_x(t) \rangle + \int_{\R} \langle i w(t), {v_0}_x \rangle \ {\rm mod} \ \pi,
\end{equation}
for any $t \in \R$. It follows from \eqref{eq:perturb} and Sobolev embedding theorem that $w \in \boC^1(\R, L^2(\R))$. Since $w$ also belongs to $\boC^0(\R, H^1(\R))$, both of the integrals in \eqref{eq:morteau} are differentiable with respect to $t$ on $\R$ and we have integrating by parts,
\begin{align*}
\frac{d}{dt} \Big( [p](v(t)) \Big)_{|_{t = s}} = \int_{\R} \langle i v_t(s), v_x(s) \rangle & = \int_{\R} \langle - v_{xx}(s) + v(s) \big(|v(s)|^2 - 1 \big), v_x(s) \rangle\\
& = \int_{\R} \partial_x \Big( - \frac{1}{2} (v_x(s))^2 + \frac{1}{4} (1 - |v(s)|^2)^2 \Big) = 0.
\end{align*}

In case $v_0 \in \boX^1 \setminus \boX^2$, we approximate $v_0$ by a sequence $(v_{0,n})_{n \in \N}$ in $\boX^2(\R)$ (e.g. by mollification) for the $H^1$-norm, and use the continuity of the flow map $w(0) \mapsto w(t)$ for \eqref{eq:perturb} from $H^1(\R)$ to $\boC^0([-T,T], H^1(\R))$ for any fixed $T > 0$ (see \cite{Gerard1,Gallo1,Gerard2}), and the continuity of $[p]$ for the $H^1$-norm.

If $v_0 \in \boZ^1\cap \boX^2$, it follows from the embedding of $H^2(\R)$ into $\boC_0^0(\R)$, that $v(t) \in \boZ^1\cap \boX^2$ for any $t \in \R$. In view of Lemma \ref{andouille}, it then suffices to replace \eqref{eq:morteau} by
$$\boP(v(t)) = \boP(v_0) + \frac{1}{2} \int_{\R} \langle i w(t), w_x(t) \rangle + \int_{\R} \langle i w(t), {v_0}_x\rangle,$$
and to repeat the argument above. When $v_0 \in \boZ^1,$ one also argues by approximation.
\end{proof}

The following is a (rigorous) localised version of the evolution law for the center of mass.

\begin{prop}
\label{prop:centredemasse}
Let $\chi \in \boC_c^\infty(\R)$ be given and let $v_0 \in \boX^1$. We denote by $v$ the solution to
\eqref{GP} with initial datum $v_0.$ Then,
\begin{equation}\label{eq:evolbari}
\frac{d}{dt} \int_\R x \big( |v(x)|^2 - 1 \big) \chi(x) dx = 2 \int_\R \langle i v(x), v_x(x) \rangle (x \chi(x))_x dx, \ \forall t \in \R.
\end{equation}
\end{prop}

\begin{proof}
Since $v \in \boC^1(\R, H^{- 1}_{\rm loc}(\R)) \cap \boC^0(\R, H^1_{\rm loc}(\R))$, we may
differentiate under the integral sign, which yields
\begin{align*}
\frac{d}{dt} \int_\R x \big( |v(x)|^2 - 1 \big) \chi(x) dx & = 2 \int_\R x \langle v(x), \partial_t v(x) \rangle \chi(x) dx\\
& = 2 \int_\R x \langle v(x), i v_{xx}(x) + i v (x) \big( 1 - |v(x)|^2 \big) \rangle \chi(x) dx\\
& = 2 \int_\R \langle i v(x), v_x(x) \rangle (x \chi(x))_x dx.
\end{align*}
\end{proof}

Formula \eqref{eq:evolbari} is particularly interesting when $\chi$ is an approximation of unity since, if one could take $\chi \equiv 1$, the right-hand side of \eqref{eq:evolbari} would represent $4 P(v)$.

\section{Proofs of Theorem \ref{stationnaire} and Theorem \ref{vitesse}}
\label{sect:5}

\begin{proof}[Proof of Theorem \ref{stationnaire}]
By contradiction, assume that there exist some numbers $\varepsilon > 0$ and
$A > 0$, a sequence $(v_{0, n})_{n \in \N}$ in $\boX^1$ verifying
\begin{equation}
\label{eq:approxi}
d_{A, \boX^1}(v_{0, n}, \v_0) \to 0, \ {\rm as} \ n \to + \infty,
\end{equation}
and a sequence of times $(t_n)_{n \in \N}$ such that
\begin{equation}
\label{eq:pasproche}
\inf_{(a, \theta) \in \R^2} d_{A,\boX^1} (v_n(\cdot + a, t_n), \exp i \theta \ \v_0(\cdot)) \geq \varepsilon,
\end{equation}
where $v_n$ is the solution to \eqref{GP} with initial datum $v_{0, n}$. It follows from \eqref{eq:approxi} and the continuity of $E$ and $[p]$ with respect to $d_{A, \boX^1}$
(see Lemma \ref{lem:lip}) that
$$[p](v_{0, n}) \to [p](\v_0) = \frac{\pi}{2}, \ {\rm and} \ E(v_{0, n}) \to E(\v_0) =
\E_{\rm min} \Big( \frac{\pi}{2} \Big), \ {\rm as} \ n \to + \infty.$$
Since $E$ and $[p]$ are conserved by the flow, we infer that
$$[p] \big( v_n(\cdot, t_n) \big) \to [p](\v_0) = \frac{\pi}{2}, \ {\rm and} \ E \big( v_n(\cdot, t_n) \big) \to E(\v_0) = \E_{\rm min} \Big( \frac{\pi}{2} \Big),$$
as $n \to + \infty$. Applying Theorem \ref{laye} to the sequence $(v_n(\cdot, t_n))_{n \in \N}$, this yields a contradiction to \eqref{eq:pasproche}.
\end{proof}

\begin{proof}[Proof of Theorem \ref{vitesse}]
Using the invariances of \eqref{GP}, we claim that it is sufficient to prove

\begin{claim}
\label{shift1}
Given any $\varepsilon > 0$ and $A > 0$, there exists some constant $K$, only depending on $A$, and some positive number $\delta > 0$ such that, if $v_0$ and $v$ are as in Theorem \ref{stationnaire} and if \eqref{pioneer} holds, then
$$|a(t)| \leq K \varepsilon,$$
for any $t \in [0, 1]$, and for any of the points $a(t)$ satisfying inequality \eqref{solaris} for some $\theta(t)\in \R$.
\end{claim}

Indeed, consider the positive numbers $\delta_{\varepsilon, A}$ provided by Claim \ref{shift1}. In view of Theorem \ref{stationnaire}, there exists some positive number $\delta$ such that, if \eqref{pioneer} holds, then, for any $t \in \R$, there exist numbers $a(t)$ and $\theta(t)$ such that
$$d_{A, \boX^1} \big( v(\cdot + a(t), t), \exp i \theta (t) \ \v_0(\cdot) \big) < \delta_{\varepsilon, A}.$$
Given any real number $t_0$, we then denote $w_0(x) = \exp(- i \theta(t_0)) \ v(x + a(t_0), t_0)$, and consider the solution $w$ to \eqref{GP} with initial datum $w_0$. It follows from the definition of $\delta_{\varepsilon, A}$ and Claim \ref{shift1} that, for any $s \in [0, 1]$, and for any numbers $\tilde{a}(s)$ and $\tilde{\theta}(s)$ such that
$$d_{A, \boX^1} \big( w(\cdot + \tilde{a}(s), s), \exp i \tilde{\theta} (s) \ \v_0(\cdot) \big) < \varepsilon,$$
we have
\begin{equation}
\label{toutpetit}
|\tilde{a}(s)| \leq K \varepsilon.
\end{equation}
On the other hand, it follows from the uniqueness of the solutions to \eqref{GP} that
$$v(x, t) = \exp i \theta(t_0) \ w \big( x - a(t_0), t - t_0 \big),$$
so that the points $a(t)$ satisfying inequality \eqref{solaris} for some $t \in [t_0, t_0 + 1]$ and some $\theta(t) \in \R$, are given by
$$a(t) = a(t_0) + \tilde{a}(t - t_0).$$
Hence, by \eqref{toutpetit},
$$|a(t) - a(t_0)| \leq K \varepsilon,$$
for any real numbers $t_0$ and $t$ such that $|t - t_0| \leq 1$, so that
$$|a(t) - a(0)| \leq K \varepsilon (1 + |t|).$$
This completes the proof of Theorem \ref{vitesse}, assuming that $\delta$ is chosen sufficiently small so that we can choose $a(0) = 0$.
\end{proof}

We finally prove Claim \ref{shift1}.

\begin{proof}[Proof of Claim \ref{shift1}]
Let $\chi \in \boC^\infty_{c}(\R, [0, 1])$ be even, and such that $\chi \equiv 1$ on $[-1, 1]$ and $\chi \equiv 0$ outside $[- 2, 2]$, and denote by $\chi_R$ the function $\chi_R(x) \equiv \chi(\frac{x}{R})$, for any $R > 1$. In view of the antisymmetry of $\v_0$ with respect to reflexion, we have
$$G_a(\v_0) \equiv \frac{1}{2 m(\v_0)} \int_{\R} x \big( |\v_0(x - a)|^2 - 1 \big) dx = a,$$
for any $a \in \R$, so that, in view of the exponential decay of $|\v_0|^2 - 1$ at infinity,
$$G_{a, R}(\v_0) \equiv \frac{1}{2 m(\v_0)} \int_{\R} x \big( |\v_0(x - a)|^2 - 1 \big) \chi_R(x) dx \to G_a(\v_0) = a,$$
as $R \to + \infty$, uniformly with respect to $a \in [- 1, 1]$. We fix $R > 1$ such that
\begin{equation}
\label{eq:jet}
|G_{a, R}(\v_0) - a| \leq \varepsilon,
\end{equation}
for any $a \in [- 1, 1]$, so that any shift $a$ may be controlled by the quantity $G_{a, R}(\v_0)$, up to some error term $\varepsilon$, provided that $|a| \leq 1$.

Since $\v_0$ is real-valued, we also have
$$\int_\R \langle i \exp i \theta \ \v_0(x - a), \partial_x \big( \exp i \theta \ \v_0(x - a) \big) \rangle \partial_x (x \chi_R(x)) dx = 0,$$
for any $(a, \theta) \in \R^2$. By Cauchy-Schwarz inequality and the definition of $d_{A,\boX^1}$, we therefore infer that there exists some constant $K \geq 1$, depending possibly on $A$, such that if
\begin{equation}
\label{eq:coucou}
d_{A,\boX^1} \big( v(\cdot +a(t), t), \exp i \theta(t) \ \v_0 \big) \leq \varepsilon,
\end{equation}
for $\varepsilon$ sufficiently small, then
\begin{equation}
\label{eq:coucou1}
\bigg| \frac{1}{m(\v_0)} \int_\R \langle i v(x, t), \partial_x v(x,t) \rangle \partial_x (x \chi_R(x)) dx \bigg| \leq K \varepsilon,
\end{equation}
and moreover,
\begin{equation}
\label{eq:coucou2}
\bigg| G_{a(t), R}(\v_0) - \frac{1}{2 m(\v_0)} \int_\R x \big( |v(x,t)|^2 - 1 \big) \chi_R(x) dx \bigg| \leq K \varepsilon.
\end{equation}
Notice that we may assume that $40 K \varepsilon \leq 1$ for $\varepsilon$ sufficiently small. Proposition \ref{prop:centredemasse} now gives
\begin{align*}
\int_\R x \big( & |v(x, t)|^2 - 1 \big) \chi_R(x) dx - \int_\R x \big( |v_0(x)|^2 - 1 \big) \chi_R(x) dx\\
& = 2 \int_0^t \int_\R \langle i v(x, s), \partial_x v(x, s) \rangle \partial_x (x \chi_R (x)) dx ds,
\end{align*}
so that, by \eqref{eq:coucou1} and \eqref{eq:coucou2},
$$\Big| G_{a(t), R}(\v_0) - G_{a(0), R}(\v_0) \Big| \leq 3 K \varepsilon.$$
Provided that $|a(t)| \leq 1$ and $|a(0)| \leq 1$, we conclude, in view of \eqref{eq:jet}, that
$$\big| a(t) - a(0) \big| \leq 5 K \varepsilon.$$
We finally choose $\delta$ so that \eqref{eq:coucou} holds for any $t \in [-1, 1] $ according to Theorem \ref{stationnaire}, and we may moreover take $a(0) = 0$. Claim \ref{shift1} follows provided that we may prove e.g. that $|a(t)| < \frac{1}{2}$ for any $t \in [0, 1]$.

Assume by contradiction that there exists some numbers $0 \leq t^* \leq 1$, $a(t^*)$ and $\theta(t^*)$ for which \eqref{eq:coucou} holds, but $|a(t^*)| \geq \frac{1}{2}$. Without loss of generality, we may assume that $t^*$ is minimal with respect to that property, so that
\begin{equation}
\label{ader1}
|a(t)| \leq 5 K \varepsilon,
\end{equation}
for any $t \in [0, t^*)$. In view of \eqref{eq:coucou}, we obtain
$$d_{A,\boX^1} \big( \exp i \theta(t^*) \ \v_0 (\cdot - a(t^*)), \exp i \theta(t) \ \v_0 (\cdot - a(t) \big) \leq d_{A,\boX^1} \big( v(t^*), v(t) \big) + 2 \varepsilon.$$
In view of the continuity of the map $t \mapsto v(t)$ with respect to the distance $d_{A,\boX^1}$ (see \cite{Gerard2}), we are led to
\begin{equation}
\label{ader2}
d_{A,\boX^1} \big( \exp i \theta(t^*) \ \v_0 (\cdot - a(t^*)), \exp i \theta(t) \ \v_0 (\cdot - a(t)) \big) \leq 3 \varepsilon,
\end{equation}
for any $t$ sufficiently close to $t^*$. On the other hand, if
$$\inf_{\theta \in \R} \int_{\R} |\v_0'(x) - \exp i \theta \ \v_0'(x - a)|^2 dx \leq 9 \varepsilon^2,$$
for $\varepsilon$ sufficiently small, then we have $|a| \leq \frac{1}{8}$, so that, by \eqref{ader1} and \eqref{ader2},
$$|a(t^*)| \leq |a(t^*) - a(t)| + |a(t)| \leq \frac{1}{8} + 5 K \varepsilon \leq \frac{1}{4},$$
which provides the required contradiction.
\end{proof}

\bibliographystyle{plain}
\bibliography{Bibliogr}

\begin{thebibliography}{10}

\bibitem{Benjami1}
T.~B. Benjamin.
\newblock The stability of solitary waves.
\newblock {\em Proc. Roy. Soc. Lond., Ser. A}, 328(1573):153--183, 1972.

\bibitem{BetGrSa2}
F.~B\'ethuel, P.~Gravejat, and J.-C. Saut.
\newblock Existence and properties of travelling waves for the
  {Gross}-{Pitaevskii} equation.
\newblock {\em Preprint}.

\bibitem{diMeGal1}
L.~di~Menza and C.~Gallo.
\newblock The black solitons of one-dimensional {NLS} equations.
\newblock {\em Nonlinearity}, 20(2):461--496, 2007.

\bibitem{Gallo1}
C.~Gallo.
\newblock The {Cauchy} problem for defocusing nonlinear {Schr\"odinger}
  equations with non-vanishing initial data at infinity.
\newblock {\em Preprint}.

\bibitem{Gerard1}
P.~G\'erard.
\newblock The {Cauchy} problem for the {Gross}-{Pitaevskii} equation.
\newblock {\em Ann. Inst. Henri Poincaré, Analyse Non Linéaire},
  23(5):765--779, 2006.

\bibitem{Gerard2}
P.~G\'erard.
\newblock The {Gross}-{Pitaevskii} equation in the energy space.
\newblock {\it Preprint}.

\bibitem{GeraZha1}
P.~G\'erard and Zhifei Zhang.
\newblock Orbital stability of traveling waves for the {1D}
  {Gross}-{Pitaevskii} equation.
\newblock {\em Preprint}.

\bibitem{GriShSt1}
M.~Grillakis, J.~Shatah, and W.A. Strauss.
\newblock Stability theory of solitary waves in the presence of symmetry {I}.
\newblock {\em J. Funct. Anal.}, 74(1):160--197, 1987.

\bibitem{KivsLut1}
Y.S. Kivshar and B.~Luther-Davies.
\newblock Dark optical solitons: physics and applications.
\newblock {\em Phys. Rep.}, 298:81--197, 1998.

\bibitem{LinZhiw1}
Zhiwu Lin.
\newblock Stability and instability of traveling solitonic bubbles.
\newblock {\em Adv. Differential Equations}, 7(8):897--918, 2002.

\bibitem{ShabZak2}
A.B. Shabat and V.E. Zakharov.
\newblock Interaction between solitons in a stable medium.
\newblock {\em Sov. Phys. JETP}, 37:823--828, 1973.

\bibitem{Zhidkov1}
P.E. Zhidkov.
\newblock {\em {Korteweg}-{De} {Vries} and nonlinear {Schr\"odinger} equations
  : qualitative theory}, volume 1756 of {\em Lecture Notes in Mathematics}.
\newblock Springer-Verlag, Berlin, 2001.

\end{thebibliography}

\end{document}